\documentclass{article}%
\usepackage{CJK}
\usepackage[CJKbookmarks,pdfstartview=FitH]{hyperref}
\usepackage{amsfonts}
\usepackage{amsmath}
\usepackage{amssymb}
\usepackage[dvipsnames,usenames]{color}
\usepackage{color}
\usepackage{colortbl}
\usepackage{graphicx}
\usepackage{amssymb}
\usepackage{graphicx}
\usepackage{makeidx}%
\providecommand{\U}[1]{\protect\rule{.1in}{.1in}}
\newtheorem{theorem}{\hskip\parindent\bf{Theorem}}[section]

\newtheorem{corollary}[theorem]{Corollary}

\newtheorem{lemma}[theorem]{Lemma}

\newtheorem{proposition}[theorem]{Proposition}

\newenvironment{proof}[1][Proof]{\noindent\textbf{#1.} }{\ \hfill\rule{0.5em}{0.5em}}

\RequirePackage{CJK}
\AtBeginDocument{\begin{CJK*}{GBK}{song}\CJKtilde}
\AtEndDocument{\end{CJK*}} \makeindex

\begin{document}

\title{The ternary Goldbach problem with primes in positive density sets}
\author{Quanli Shen \thanks{The research was supported by 973Grant 2013CB834201.}}
\date{}
\maketitle

\begin{abstract}
Let $\mathcal{P}$ denote the set of all primes. $P_{1},P_{2},P_{3}$ are three
subsets of $\mathcal{P}$. Let $\underline{\delta}(P_{i})$ $(i=1,2,3)$ denote
the lower density of $P_{i}$ in $\mathcal{P}$, respectively. It is proved that
if $\underline{\delta}(P_{1})>5/8$, $\underline{\delta}(P_{2})\geq5/8$, and
$\underline{\delta}(P_{3})\geq5/8$, then for every sufficiently large odd
integer n, there exist $p_{i} \in P_{i}$ such that $n=p_{1}+p_{2}+p_{3}$. The
condition is the best possible.
\end{abstract}

\setcounter{tocdepth}{5}

\section{Introduction}

The ternary Goldbach conjecture states that every odd positive integer greater
than 5 can be written as sums of three primes. It was first proposed from an
exchange of letters between Goldbach and Euler in 1742. Until 1923, Hardy and
Littlewood \cite{11 G. H. Hardy and J. E. Littlewood} claimed it is true for
sufficiently large positive odd integers, depending on the generalised Riemann
hypothesis (GRH). Instead, in 1937, I. M. Vinogradov \cite{7 Vinogradov}
showed for the first time a nontrivial estimate of exponential sums over
primes, and solved this problem unconditionally. It should be noted that,
recently, H. A. Helfgott \cite{8 Helfgott,9 Helfgott major,10 Helfgott
ternary} (2014) has completely proved the ternary Goldbach conjecture for
every odd integer n greater than 5.

The main idea used above is circle method which is founded by Hardy and
Littlewood. On the other hand, B. Green proposed the transference principle,
and now it is playing an increasing important role in number theory \cite{2
Green,3 Green and Tao}. Employing this method, H. Li and H. Pan extended
\cite{4 H. Li and H. Pan} (2010) the Vinogradov's three primes theorem to a
density version. Let $\mathcal{P}$ denote the set of all primes. For a subset
$A\subset\mathcal{P}$, the lower density of $A$ in $\mathcal{P}$ is defined
by
\[
\underline{\delta}(A)=\liminf\limits_{N\rightarrow\infty}\frac{\left\vert
A\cap\lbrack1,N]\right\vert }{\left\vert \mathcal{P}\cap\lbrack1,N]\right\vert
}.
\]
They stated that if $P_{1},P_{2},P_{3}$ are three subsets of $\mathcal{P}$
satisfying that
\[
\underline{\delta}(P_{1})+\underline{\delta}(P_{2})+\underline{\delta}%
(P_{3})>2,
\]
then for every sufficiently large odd integer $n$, there exist $p_{i}\in
P_{i}$ $(i=1,2,3)$ such that $n=p_{1}+p_{2}+p_{3}$.

Motivated by the work of Li and Pan, X. Shao proved \cite{5 X. Shao} (2014)
that if $A$ is a subset of $\mathcal{P}$ with
\[
\underline{\delta}(A)>\frac{5}{8},
\]
then for every sufficiently large odd integer $n$, there exist $p_{i}\in A$
$(i=1,2,3)$ such that $n=p_{1}+p_{2}+p_{3}$. It is worth mentioning that X.
Shao gave \cite{6 X. Shao} (2014) an l-function-free proof of Vinogradov's
three primes theorem.

This paper is to revise Shao's method, and show the following result.

\begin{theorem}
\label{thm 1.1}Let $P_{1},P_{2},P_{3}$ be three subsets of $\mathcal{P}$,
satisfying that%
\[
\underline{\delta}(P_{1})>\frac{5}{8}\text{, }\underline{\delta}(P_{2}%
)\geq\frac{5}{8}\text{, }\underline{\delta}(P_{3})\geq\frac{5}{8}.
\]
Then for every sufficiently large odd integer $n$, there exist $p_{i}\in
P_{i}$ $(i=1,2,3)$ such that $n=p_{1}+p_{2}+p_{3}$.
\end{theorem}

Note that Theorem 1.1 in \cite{5 X. Shao} can be immediately obtained from the
above theorem. We remark that the condition in Theorem \ref{thm 1.1} cannot be
improved, and the counterexample can be seen in \cite{5 X. Shao}. Here we
provide another counterexample. Let $P_{1}=P_{2}=P_{3}=\{n\in P|n\equiv
1,4,7,11,13$ $(\operatorname{mod}15)\}$. Note that $\underline{\delta}%
(P_{1})=\underline{\delta}(P_{2})=\underline{\delta}(P_{2})=5/8$, but
$N\equiv2$ $(\operatorname{mod}15)$ cannot be written by $p_{1}+p_{2}+p_{3}$
with $p_{i}\in P_{i}$ $(i=1,2,3)$.

The key to our proof is the following theorem:

\begin{theorem}
\label{thm 1.2}Let $n\geq6$ be an even number. Let $\{a_{i}\},\{b_{i}%
\},\{c_{i}\}$ $(0\leq i<n)$ are three decreasing sequences of real numbers in
$[0,1]$. Let $A,B,C$ denote the averages of $\{a_{i}\},\{b_{i}\},\{c_{i}\}$,
respectively. Suppose that for all triples $(i,j,k)$ with $0\leq i,j,k<n$ and
$i+j+k\geq n$, we have
\[
a_{i}b_{j}+b_{j}c_{k}+c_{k}a_{i}\leq\frac{5}{8}(a_{i}+b_{j}+c_{k}).
\]
Then we have
\[
AB+BC+CA\leq\frac{5}{8}(A+B+C).
\]

\end{theorem}

It was \cite[Lemma 2.2]{5 X. Shao} with the condition $n\geq10$, which could
only deduce Theorem \ref{thm 1.1} with $P_{i}=A$ $(i=1,2,3)$. X. Shao remarked
there exists the numerical evidence for the conditon $n\geq6$. In this paper,
we verify its truth and apply it as the critical step which enables the
argument of Shao to be valid for the general case.

Theorem \ref{thm 1.2} can deduce the following

\begin{theorem}
\label{thm 1.3}Let $0<\delta<5/32$ and $0<\eta<2\delta/5$ be parameters. Let
$m$ be a square-free positive odd integer. Let $f_{1},f_{2},f_{3}%
:\mathbb{Z}_{m}^{\ast}\rightarrow\lbrack0,1]$ be functions satisfying%
\[
\frac{1}{\phi(m)}%
{\displaystyle\sum\limits_{x\in\mathbb{Z}_{m}^{\ast}}}
f_{1}(x)>\frac{5}{8}+\delta\text{, }\frac{1}{\phi(m)}%
{\displaystyle\sum\limits_{x\in\mathbb{Z}_{m}^{\ast}}}
f_{2}(x)>\frac{5}{8}-\eta\text{, }\frac{1}{\phi(m)}%
{\displaystyle\sum\limits_{x\in\mathbb{Z}_{m}^{\ast}}}
f_{3}(x)>\frac{5}{8}-\eta,
\]
where $\phi$ is the Euler totient function. Then for any $x\in\mathbb{Z}_{m}$,
there exist $a,b,c\in\mathbb{Z}_{m}^{\ast}$ with $x=a+b+c$ such that
\[
f_{1}(a)f_{2}(b)f_{3}(c)>0,\text{ }f_{1}(a)+f_{2}(b)+f_{3}(c)>\frac{3}{2}.
\]

\end{theorem}

Theorem \ref{thm 1.3} is crucial for applying transference principle in
section 4. It also asserts that $A+B+C$ must cover all residue classes modulo
$m$ for any square-free odd $m$, provided that $A,B,C\subset\mathbb{Z}%
_{m}^{\ast}$ with $\delta(A)>5/8,\delta(B)\geq5/8,\delta(C)\geq5/8$, where
$\delta(A)$ denotes the density of $A$ in $\mathbb{Z}_{m}^{\ast}$. It is the
following Corollary \ref{corol 1.4}, which extends \cite[Corollary 1.5]{5 X.
Shao}. Note that if $m$ is a prime, Corollary \ref{corol 1.4} can be
immediately proved by the Cauchy-Davenport-Chowla theorem \cite{12 Green and
Vu}, which asserts that if $A,B,C$ are subsets of $\mathbb{Z}_{p}$ for prime
$p$, then $\left\vert A+B+C\right\vert $ $\geq\min(\left\vert A\right\vert
+\left\vert B\right\vert +\left\vert C\right\vert -2,p)$. However, we cannot
assure whether the Cauchy-Davenport-Chowla theorem is valid for arbitrary
integer $m$.

If $A,B,C\subset\mathbb{Z}_{m}^{\ast}$ are subsets of $\mathbb{Z}_{m}^{\ast}$,
denote by $f_{i}(x)$ $(x=1,2,3)$ the characteristic functions of $A,B,C$,
respectively. Then by Theorem \ref{thm 1.3} we have the following

\begin{corollary}
\label{corol 1.4}Let $m$ be a square-free positive odd integer. Let
$A_{1},A_{2},A_{3}$ be three subsets of $\mathbb{Z}_{m}^{\ast}$ with
$\left\vert A_{1}\right\vert >\frac{5}{8}\phi(m),$ and $\left\vert
A_{i}\right\vert \geq\frac{5}{8}\phi(m)$ $(i=2,3)$. Then $A_{1}+A_{2}%
+A_{3}=\mathbb{Z}_{m}$.
\end{corollary}

\section{Proof of Theorem 1.2}

We first make the change of the variables $x_{i}=\frac{16}{5}a_{i}%
-1,\ y_{i}=\frac{16}{5}b_{i}-1,\ z_{i}=\frac{16}{5}c_{i}-1$. Note that
$\{x_{i}\},\{y_{i}\},\{z_{i}\}$ are three decreasing sequences of real numbers
in $[-1,2.2]$. Let $X,Y,Z$ denote the averages of $\{x_{i}\},\{y_{i}%
\},\{z_{i}\}$, respectively.

Now our goal is to confirm that if
\begin{equation}
x_{i}y_{j}+y_{j}z_{k}+z_{k}x_{i}\leq3 \label{equ:2.1}%
\end{equation}
for all $0\leq i,j,k<n$ with $i+j+k\geq n$, then
\[
XY+YZ+ZX\leq3.
\]

Write $n=2m$ and
\[
X_{0}=%
{\displaystyle\sum\limits_{i=0}^{m-1}}
x_{i},X_{1}=%
{\displaystyle\sum\limits_{i=m}^{2m-1}}
x_{i},Y_{0}=%
{\displaystyle\sum\limits_{i=0}^{m-1}}
y_{i},Y_{1}=%
{\displaystyle\sum\limits_{i=m}^{2m-1}}
y_{i},Z_{0}=%
{\displaystyle\sum\limits_{i=0}^{m-1}}
z_{i},Z_{1}=%
{\displaystyle\sum\limits_{i=m}^{2m-1}}
z_{i}.
\]
Define $\mathcal{M=}\{(i,j,k)|0\leq i,j<m,m\leq k\leq n-1,i+j+k\equiv0$
$(\operatorname{mod}m)\}$. Note that all of the elements in $\mathcal{M}$
except $(0,0,m)$ satisfy (\ref{equ:2.1}), and $\#(\mathcal{M})=m^{2}$. We have%
\[%
{\displaystyle\sum\limits_{(i,j,k)\in\mathcal{M}}}
(x_{i}y_{j}+y_{j}z_{k}+z_{k}x_{i})-(x_{0}y_{0}+y_{0}z_{m}+z_{m}x_{0}%
)\leq3(m^{2}-1).
\]
Noting also that if two of the variables $i,j,k$ are fixed, then the third is
uniquely determined by the condition $i+j+k\equiv0$ $(\operatorname{mod}m)$.
Thus, we have
\[%
{\displaystyle\sum\limits_{(i,j,k)\in\mathcal{M}}}
(x_{i}y_{j}+y_{j}z_{k}+z_{k}x_{i})=X_{0}Y_{0}+Y_{0}Z_{1}+Z_{1}X_{0}.
\]
It follows that
\[
X_{0}Y_{0}+Y_{0}Z_{1}+Z_{1}X_{0}\leq3(m^{2}-1)+(x_{0}y_{0}+y_{0}z_{m}%
+z_{m}x_{0}).
\]
Similarly,%
\begin{align*}
X_{0}Y_{1}+Y_{1}Z_{0}+Z_{0}X_{0}  &  \leq3(m^{2}-1)+(x_{0}y_{m}+y_{m}%
z_{0}+z_{0}x_{0}),\\
X_{1}Y_{0}+Y_{0}Z_{0}+Z_{0}X_{1}  &  \leq3(m^{2}-1)+(x_{m}y_{0}+y_{0}%
z_{0}+z_{0}x_{m}).
\end{align*}
By the above three inequalities, we claim that%
\begin{align}
&  n^{2}(XY+YZ+ZX)\nonumber\\
&  =(X_{0}+X_{1})(Y_{0}+Y_{1})+(Y_{0}+Y_{1})(Z_{0}+Z_{1})+(Z_{0}+Z_{1}%
)(X_{0}+X_{1})\nonumber\\
&  \leq9(m^{2}-1)+(x_{0}y_{0}+y_{0}z_{m}+z_{m}x_{0})+(x_{0}y_{m}+y_{m}%
z_{0}+z_{0}x_{0})+\nonumber\\
&  +(x_{m}y_{0}+y_{0}z_{0}+z_{0}x_{m})+X_{1}Y_{1}+Y_{1}Z_{1}+Z_{1}X_{1}.
\label{equ:2.2}%
\end{align}
For convenience, write%
\begin{align*}
U  &  =(x_{0}+x_{m})(y_{0}+y_{m})+(y_{0}+y_{m})(z_{0}+z_{m})+(z_{0}%
+z_{m})(x_{0}+x_{m}),\\
\Delta_{0}  &  =x_{0}y_{0}+y_{0}z_{0}+z_{0}x_{0},\\
\Delta_{m}  &  =x_{m}y_{m}+y_{m}z_{m}+z_{m}x_{m},\\
\Delta_{m,0}  &  =x_{m}z_{0}+y_{m}z_{0}+y_{m}x_{0}+z_{m}x_{0}+x_{m}y_{0}%
+z_{m}y_{0}.
\end{align*}
Then (\ref{equ:2.2}) can be denoted by
\begin{equation}
n^{2}(XY+YZ+ZX)\leq9(m^{2}-1)+U-\Delta_{m}+X_{1}Y_{1}+Y_{1}Z_{1}+Z_{1}X_{1}.
\label{equ:2.3}%
\end{equation}
It follows from (\ref{equ:2.1}) that
\begin{align*}
x_{m}y_{m}+y_{m}z_{0}+z_{0}x_{m}  &  \leq3,\\
x_{m}y_{0}+y_{0}z_{m}+z_{m}x_{m}  &  \leq3,\\
x_{0}y_{m}+y_{m}z_{m}+z_{m}x_{0}  &  \leq3.
\end{align*}
Then we have
\begin{equation}
\Delta_{m,0}\leq9-\Delta_{m}. \label{equ:2.4}%
\end{equation}
Together with (\ref{equ:2.2}), we have%
\begin{equation}
n^{2}(XY+YZ+ZX)\leq9m^{2}+\Delta_{0}-\Delta_{m}+X_{1}Y_{1}+Y_{1}Z_{1}%
+Z_{1}X_{1}. \label{equ:2.5}%
\end{equation}

In fact, we will apply inequalities (\ref{equ:2.3}) and (\ref{equ:2.5})
repeatedly later.

Define $\mathcal{M}^{^{\prime}}\mathcal{=}\{(i,j,k):m\leq i,j,k<n,i+j+k\equiv
0$ $(\operatorname{mod}m)\}$. It follows from (\ref{equ:2.1}) that
\[%
{\displaystyle\sum\limits_{(i,j,k)\in\mathcal{M}^{^{\prime}}}}
(x_{i}y_{j}+y_{j}z_{k}+z_{k}x_{i})-(x_{m}y_{m}+y_{m}z_{m}+z_{m}x_{m}%
)\leq3(m^{2}-1).
\]
As has been done previously, we can deduce that%
\begin{align}
&  X_{1}Y_{1}+Y_{1}Z_{1}+Z_{1}X_{1}-\Delta_{m}\nonumber\\
&  ={\sum\limits_{(i,j,k)\in\mathcal{M}^{^{\prime}}}}(x_{i}y_{j}+y_{j}%
z_{k}+z_{k}x_{i})-(x_{m}y_{m}+y_{m}z_{m}+z_{m}x_{m})\nonumber\\
&  \leq3(m^{2}-1). \label{equ:2.6}%
\end{align}
Write $r=x_{0}+x_{m}$, $s=y_{0}+y_{m}$, $t=z_{0}+z_{m}$. We may assume that
$r+s\geq0$, $s+t\geq0$, $t+r\geq0$. In fact, if at least one is negative, say
$r+s<0$, then
\begin{equation}
U=rs+st+tr\leq rs-2(r+s)=(r-2)(s-2)-4\leq(-4)\times(-4)-4=12. \label{equ:2.7}%
\end{equation}
Note that (\ref{equ:2.3}), (\ref{equ:2.6}), and (\ref{equ:2.7}) together can
deduce $XY+YZ+ZX\leq3$. It means the lemma has been true. Hence, we only need
to consider the case $r+s\geq0$, $s+t\geq0$, $t+r\geq0$. We can see that $U$
is an increasing function with the variables $r,s,t$.

We next consider four cases.\newline

\textbf{Case 1.} If $X_{1},Y_{1},Z_{1}<0$. Considering the inequality
(\ref{equ:2.3}), we note that $X_{1}Y_{1}+Y_{1}Z_{1}+Z_{1}X_{1}$ is decreasing
with the variables $X_{1},Y_{1},Z_{1}$. Then we have
\begin{align*}
&  X_{1}Y_{1}+Y_{1}Z_{1}+Z_{1}X_{1}\\
&  \leq\lbrack x_{m}-(m-1)][y_{m}-(m-1)]+[y_{m}-(m-1)][z_{m}-(m-1)]+\\
&  +[z_{m}-(m-1)][x_{m}-(m-1)]\\
&  \leq3(m-1)^{2}-2(m-1)(x_{m}+y_{m}+z_{m})+\Delta_{m}.
\end{align*}
Since $U$ is increasing, we have
\begin{align*}
U  &  \leq(2.2+x_{m})(2.2+y_{m})+(2.2+y_{m})(2.2+z_{m})+(2.2+z_{m}%
)(2.2+x_{m})\\
&  \leq14.52+4.4(x_{m}+y_{m}+z_{m})+\Delta_{m}.
\end{align*}
Together with $\Delta_{m}\leq3$ by (\ref{equ:2.1}), we have
\[
U-\Delta_{m}+X_{1}Y_{1}+Y_{1}Z_{1}+Z_{1}X_{1}\leq17.52+3(m-1)^{2}%
-(2m-6.4)(x_{m}+y_{m}+z_{m}).
\]
If $m=3$, we bound the term $x_{m}+y_{m}+z_{m}$ by $2.2\times3$ trivially.
Then
\[
U-\Delta_{m}+X_{1}Y_{1}+Y_{1}Z_{1}+Z_{1}X_{1}\leq3m^{2}-19m+63\leq3m^{2}+6.
\]
If $m\geq4$, note that the term $x_{m}+y_{m}+z_{m}$ is greater than
$-1\times3$. Then
\[
U-\Delta_{m}+X_{1}Y_{1}+Y_{1}Z_{1}+Z_{1}X_{1}\leq3m^{2}+2.
\]
Hence, it follows from (\ref{equ:2.3}) that $XY+YZ+ZX\leq3$ for all $m\geq
3$.\newline

\textbf{Case 2. }If exactly two of $X_{1},Y_{1},Z_{1}$ are negative, say
$X_{1}<0$, $Y_{1}<0$, and $Z_{1}\geq0$. Now we consider the inequality
(\ref{equ:2.5}). Since $Y_{1}Z_{1},Z_{1}X_{1}$ are both nonpositve, we have
\[
X_{1}Y_{1}+Y_{1}Z_{1}+Z_{1}X_{1}\leq X_{1}Y_{1}.
\]
Noting that $X_{1},Y_{1}<0$, then $X_{1}Y_{1}$ is trivially bounded by
$[x_{m}-(m-1)][y_{m}-(m-1)]$. Hence,
\begin{align*}
&  X_{1}Y_{1}+Y_{1}Z_{1}+Z_{1}X_{1}-\Delta_{m}\\
&  \leq-(m-1)(x_{m}+y_{m})+(m-1)^{2}-(y_{m}z_{m}+z_{m}x_{m})\\
&  \leq2(m-1)+(m-1)^{2}-z_{m}(y_{m}+x_{m})\\
&  \leq m^{2}-1+2\times2.2.
\end{align*}
The second inequality above holds since $z_{m}\geq0$ when $Z_{1}\geq0$.
Together with $\Delta_{0}\leq3\times2.2^{2}$ and (\ref{equ:2.5}), we have
$n^{2}(XY+YZ+ZX)\leq10m^{2}+18\leq12m^{2}$ $(m\geq3)$. Hence, we have
$XY+YZ+ZX\leq3$.\newline

\textbf{Case 3.} If exactly one of $X_{1},Y_{1},Z_{1}$ are negative, say
$X_{1}<0$, $Y_{1}\geq0$, and $Z_{1}\geq0$. And suppose at least one of $X_{1}+
$ $Y_{1}$ and $X_{1}+Z_{1}$ is negative. We may assume $X_{1}+$ $Y_{1}<0$.
Since the term $X_{1}Y_{1}+Y_{1}Z_{1}+Z_{1}X_{1}=(X_{1}+Y_{1})Z_{1}+X_{1}%
Y_{1}\leq0$, we can ignore it in (\ref{equ:2.5}). Noting that at most two
terms of $\Delta_{m}$ are nonpostive, we have $-\Delta_{m}\leq2.2^{2}\times2$.
Together with $\Delta_{0}\leq3\times2.2^{2}$, it follows that $n^{2}%
(XY+YZ+ZX)\leq9m^{2}+5\times2.2^{2}\leq12m^{2}$ $(m\geq3)$. This leads to
$XY+YZ+ZX\leq3$.\newline

\textbf{Case 4.} If $X_{1}+$ $Y_{1}$, $Y_{1}+Z_{1}$, and $Z_{1}+X_{1}$ are all
nonnegative. Therefore, $x_{m}+y_{m}$, $y_{m}+z_{m}$, and $z_{m}+x_{m}$ are
all nonnegative. Noting that $X_{1}Y_{1}+Y_{1}Z_{1}+Z_{1}X_{1}$ is increasing
with variables $X_{1},Y_{1},Z_{1}$, we have
\begin{equation}
X_{1}Y_{1}+Y_{1}Z_{1}+Z_{1}X_{1}\leq m^{2}\Delta_{m}. \label{equ:2.8}%
\end{equation}
Write $E=x_{0}+y_{0}-5(x_{m}+y_{m})$, $F=y_{0}+z_{0}-5(y_{m}+z_{m})$,
$G=z_{0}+x_{0}-5(z_{m}+x_{m})$. Four more cases are considered below:\newline

\textbf{(i)} Suppose $E,F,G$ are all negative. Note that
\[
\lbrack x_{0}+y_{0}-5(x_{m}+y_{m})](z_{0}-z_{m})\leq0.
\]
Upon expanding, it follows that%
\[
(x_{0}z_{0}+y_{0}z_{0})+5(x_{m}z_{m}+y_{m}z_{m})\leq x_{0}z_{m}+y_{0}%
z_{m}+5(x_{m}z_{0}+y_{m}z_{0}).
\]
Similarly, we have
\begin{align*}
(y_{0}x_{0}+z_{0}x_{0})+5(y_{m}x_{m}+z_{m}x_{m})  &  \leq y_{0}x_{m}%
+z_{0}x_{m}+5(y_{m}x_{0}+z_{m}x_{0}),\\
(z_{0}y_{0}+x_{0}y_{0})+5(z_{m}y_{m}+x_{m}y_{m})  &  \leq z_{0}y_{m}%
+x_{0}y_{m}+5(z_{m}y_{0}+x_{m}y_{0}).
\end{align*}
Combining the inequalities above, we have
\[
\Delta_{0}+5\Delta_{m}\leq3\Delta_{m,0}.
\]
Together with (\ref{equ:2.4}), we have
\begin{equation}
\Delta_{0}+8\Delta_{m}\leq27. \label{equ:2.9}%
\end{equation}
Noting that $\Delta_{m}\leq3$ by (\ref{equ:2.1}), (\ref{equ:2.5}),
(\ref{equ:2.8}), and (\ref{equ:2.9}) together can deduce that
\begin{align*}
n^{2}(XY+YZ+ZX)  &  \leq9m^{2}+\Delta_{0}+(m^{2}-1)\Delta_{m}\\
&  \leq9m^{2}+(m^{2}-9)\Delta_{m}+27\\
&  \leq12m^{2}%
\end{align*}
for $m\geq3$, which implies $XY+YZ+ZX\leq3$.\newline

\textbf{(ii) }If exactly two of $E,F,G$ are negative, say $E,F<0$, and
$G\geq0$. We can see that
\[
\lbrack x_{0}+y_{0}-5(x_{m}+y_{m})][z_{0}+x_{0}-5(z_{m}+x_{m})]\leq0.
\]
Upon expanding, we have%
\[
\Delta_{0}+25\Delta_{m}+(x_{0}-5x_{m})^{2}\leq5\Delta_{m,0},
\]
which implies that $\Delta_{0}+25\Delta_{m}\leq5\Delta_{m,0}$. Combining it
with (\ref{equ:2.4}), we have
\[
\Delta_{0}+30\Delta_{m}\leq45.
\]
Then we have
\[
\Delta_{0}+(m^{2}-1)\Delta_{m}\leq\frac{3(m^{2}-1)}{2}+\frac{31-m^{2}}%
{30}\Delta_{0}.
\]
For $3\leq m\leq5$, we have
\begin{align*}
\Delta_{0}+(m^{2}-1)\Delta_{m}  &  \leq\frac{3(m^{2}-1)}{2}+\frac{31-m^{2}%
}{30}\times2.2^{2}\times3\\
&  \leq1.1m^{2}+14\leq3m^{2}.
\end{align*}
For $m\geq6$, we have%
\begin{align*}
\Delta_{0}+(m^{2}-1)\Delta_{m}  &  \leq\frac{3(m^{2}-1)}{2}-\frac{31-m^{2}%
}{30}\times2.2^{2}\times3\\
&  \leq2m^{2}-16\leq3m^{2}.
\end{align*}
Together with (\ref{equ:2.5}) and (\ref{equ:2.8}), we have $n^{2}%
(XY+YZ+ZX)\leq12m^{2}$, which leads to $XY+YZ+ZX\leq3$.\newline

\textbf{(iii) }If exactly one of $E,F,G$ is negative, say $E<0$, $F\geq0$, and
$G\geq0$. The proof is similar to the case (ii).\newline

\textbf{(iv) }If $E,F,G$ are all nonnegative. Note that $x_{m}y_{m}\leq
(\frac{x_{m}+y_{m}}{2})^{2}$, $x_{0}+y_{0}\geq5(x_{m}+y_{m})$ and $x_{m}%
+y_{m}\geq0$ by $X_{1}+Y_{1}\geq0$, then we have%
\[
x_{m}y_{m}\leq(\frac{x_{0}+y_{0}}{10})^{2}\leq0.44^{2}.
\]
Similarly, we have $y_{m}z_{m}\leq0.44^{2}$ and $z_{m}x_{m}\leq0.44^{2}$. It
implies that $\Delta_{m}\leq3\times0.44^{2}\leq1$. We have trivially
$\Delta_{0}\leq3\times2.2^{2}$. By (\ref{equ:2.5}) and (\ref{equ:2.8}), we
have $n^{2}(XY+YZ+ZX)\leq10m^{2}+14\leq12m^{2}$ which implies $XY+YZ+ZX\leq3$.
This completes the proof.

Here we remark that for $n\geq6$, the constant $5/8$ can be slightly improved.

\section{Proof of Theorem 1.3}

The argument of the proof is similar to that in \cite{5 X. Shao}. Using
Theorem \ref{thm 1.2} we can show that

\begin{lemma}
\label{lem 3.1}Let $0<\delta<5/32$ and $0<\eta<2\delta/5$ be parameters. Let
$m$ be a square-free positive integer with $(m,30)=1$. Let $f_{1},f_{2}%
,f_{3}:\mathbb{Z}_{m}^{\ast}\rightarrow\lbrack0,1]$ satisfy
\[
\frac{1}{\phi(m)}%
{\displaystyle\sum\limits_{x\in\mathbb{Z}_{m}^{\ast}}}
f_{1}(x)>\frac{5}{8}+\delta\text{, }\frac{1}{\phi(m)}%
{\displaystyle\sum\limits_{x\in\mathbb{Z}_{m}^{\ast}}}
f_{2}(x)>\frac{5}{8}-\eta\text{, }\frac{1}{\phi(m)}%
{\displaystyle\sum\limits_{x\in\mathbb{Z}_{m}^{\ast}}}
f_{3}(x)>\frac{5}{8}-\eta.
\]
Then for every $x\in\mathbb{Z}_{m}$, there exist $a,b,c\in\mathbb{Z}_{m}%
^{\ast}$ with $x=a+b+c$, such that%
\[
f_{1}(a)f_{2}(b)+f_{2}(b)f_{3}(c)+f_{3}(c)f_{1}(a)>\frac{5}{8}(f_{1}%
(a)+f_{2}(b)+f_{3}(c)).
\]

\end{lemma}

\begin{proof}
The proof will proceed by induction. First consider the base case when $m=p$
is prime. It could prove the conclusion only for $p\geq11$ while $f_{1},f_{2}%
,f_{3}$ might be different \cite[Proposition 3.1]{5 X. Shao} and for $p\geq7$
with the constraint condition  $f_{1}=f_{2}=f_{3}$. Now\ by Theorem
\ref{thm 1.2}, we are able to show the case that $f_{1},f_{2},f_{3}$ need not
to be the same for $p\geq7$. Let $a_{0}\geq a_{1}\geq\cdots\geq a_{p-2}$ be $p-1$ values of $f_{1}(x)$
$(x\in\mathbb{Z}_{p}^{\ast})$ in decreasing order. Similarly, define
$b_{0}\geq b_{1}\geq\cdots\geq b_{p-2}$ for $f_{2}(x)$ $(x\in\mathbb{Z}%
_{p}^{\ast})$, and $c_{0}\geq c_{1}\geq\cdots\geq c_{p-2}$ for $f_{3}(x)$
$(x\in\mathbb{Z}_{p}^{\ast})$. Let $A,B,C$ denote the averages of
$\{a_{i}\},\{b_{i}\},\{c_{i}\}$, respectively. We can deduce that
\[
AB+BC+CA>\frac{5}{8}(A+B+C).
\]
To prove it, we make the change of the variables $X=\frac{16}{5}A-1$,
$Y=\frac{16}{5}B-1$, and $Z=\frac{16}{5}C-1$. Then our aim is to prove
$XY+YZ+ZX>3$ when
\[
X>1+\frac{16}{5}\delta,\text{ }Y>1-\frac{16}{5}\eta,\text{ }Z>1-\frac{16}%
{5}\eta.
\]
Note that
\begin{align*}
&  XY+YZ+ZX\\
&  >2(1+\frac{16}{5}\delta)(1-\frac{16}{5}\eta)+(1-\frac{16}{5}\eta)^{2}\\
&  =3+\frac{32}{5}\delta+(\frac{16}{5})^{2}\eta^{2}-\frac{2\times16^{2}}%
{5^{2}}\delta\eta-\frac{64}{5}\eta\\
&  >3+(\frac{16}{5})^{2}\eta^{2}-\frac{32^{2}}{5^{3}}\delta^{2}+\frac
{32}{5^{2}}\delta\\
&  >3\text{ }(0<\delta<\frac{5}{32}).
\end{align*}
Then, by Theorem \ref{thm 1.2}, there exist $0\leq i,j,k\leq p-1$
with$\ i+j+k\geq p-1$, such that
\begin{equation}
a_{i}b_{j}+b_{j}c_{k}+c_{k}a_{i}>\frac{5}{8}(a_{i}+b_{j}+c_{k}%
).\label{equ:3.2}%
\end{equation}
Define $I,J,K\subset\mathbb{Z}_{p}^{\ast}$,
\[
I=\{x:f_{1}(x)\geq a_{i}\},\text{ }J=\{x:f_{2}(x)\geq b_{j}\},\text{
}K=\{x:f_{3}(x)\geq c_{k}\}.
\]
Since $\{a_{i}\},\{b_{i}\},\{c_{i}\}$ are decreasing, we have
\begin{equation}
\left\vert I\right\vert +\left\vert J\right\vert +\left\vert K\right\vert
\geq(i+1)+(j+1)+(k+1)\geq p+2\mathbf{.}\label{equ:3.2.1}%
\end{equation}
By the Cauchy-Davenport-Chowla theorem, it follows from (\ref{equ:3.2.1})
that
\[
I+J+K=\mathbb{Z}_{p}.
\]
That means for any $x\in\mathbb{Z}_{p}$, there exist $a\in I,b\in J,c\in K$
such that $x=a+b+c$. From the definition of $I,J,K$, we can see that
\[
f_{1}(a)\geq a_{i},\text{ }f_{2}(b)\geq b_{j},\text{ }f_{3}(c)\geq c_{k}.
\]
Write $h(x,y,z)=xy+yz+zx-\frac{5}{8}(x+y+z)$. Note that $h(x,y,z)$ is
increasing with variables $x,y,z$ on the area
\[
D=\{0\leq x,y,z\leq1:x+y\geq\frac{5}{8},y+z\geq\frac{5}{8},z+x\geq\frac{5}%
{8}\}.
\]
In fact, (\ref{equ:3.2}) implies $a_{i}+b_{j}\geq\frac{5}{8}$. Otherwise
$a_{i}c_{k}+b_{j}c_{k}\leq\frac{5}{8}c_{k}$, then $a_{i}b_{j}>\frac{5}%
{8}(a_{i}+b_{j})$. But it is impossible since $0\leq a_{i},b_{j}\leq1$.
Similarly, we have $b_{j}+c_{k}\geq\frac{5}{8}$, and $c_{k}+a_{i}\geq\frac
{5}{8}$. Hence, we have%
\[
h(f_{1}(a),f_{2}(b),f_{3}(c))\geq h(a_{i},b_{j},c_{k})>0,
\]
which implies%
\[
f_{1}(a)f_{2}(b)+f_{2}(b)f_{3}(c)+f_{3}(c)f_{1}(a)>\frac{5}{8}(f_{1}%
(a)+f_{2}(b)+f_{3}(c)).
\]
Now we consider $m$ is composite and write $m=m^{^{\prime}}p$ with $p\geq7$.
Noting that $\mathbb{Z}_{m}\cong\mathbb{Z}_{m^{^{\prime}}}\times\mathbb{Z}%
_{p}$, we define $f_{i}^{^{\prime}}:\mathbb{Z}_{m^{^{\prime}}}^{\ast
}\rightarrow\lbrack0,1]$ $(i=1,2,3)$ by
\[
f_{i}^{^{\prime}}(x)=\frac{1}{p-1}%
{\displaystyle\sum\limits_{y\in\mathbb{Z}_{p}^{\ast}}}
f_{i}(x,y).
\]
Then by induction hypothesis, for any $x\in\mathbb{Z}_{m^{^{\prime}}}$, there
exists $a,b,c\in\mathbb{Z}_{m^{^{\prime}}}^{\ast}$ with $x=a+b+c$, such that
\[
f_{1}^{^{\prime}}(a)f_{2}^{^{\prime}}(b)+f_{2}^{^{\prime}}(b)f_{3}^{^{\prime}%
}(c)+f_{3}^{^{\prime}}(c)f_{1}^{^{\prime}}(a)>\frac{5}{8}(f_{1}^{^{\prime}%
}(a)+f_{2}^{^{\prime}}(b)+f_{3}^{^{\prime}}(c)).
\]
Define $a_{0}\geq a_{1}\geq\cdots\geq a_{p-2}$ be $p-1$ values of $f_{1}(a,x)$
$(x\in\mathbb{Z}_{p}^{\ast})$ in decreasing order, and similarly $\{b_{i}\}$
for $f_{2}(b,x)$ and $\{c_{i}\}$ for $f_{3}(c,x)$. Noting that the averages of
$\{a_{i}\},\{b_{i}\},\{c_{i}\}$ are $f_{1}^{^{\prime}}(a),f_{2}^{^{\prime}%
}(b),f_{3}^{^{\prime}}(c)$, respectively. It follows from Theorem
\ref{thm 1.2} that there exist $0\leq i,j,k\leq p-1$ with$\ i+j+k\geq p-1$,
such that
\[
a_{i}b_{j}+b_{j}c_{k}+c_{k}a_{i}>\frac{5}{8}(a_{i}+b_{j}+c_{k}).
\]
Similarly, we can deduce that for any $y\in\mathbb{Z}_{p}$, there exist
$u,v,w\in\mathbb{Z}_{p}^{\ast}$ with $y=u+v+w$, such that
\[
f_{1}(a,u)f_{2}(b,v)+f_{2}(b,v)f_{3}(c,w)+f_{3}(c,w)f_{1}(a,u)>\frac{5}%
{8}(f_{1}(a,u)+f_{2}(b,v)+f_{3}(c,w)).
\]
This completes the proof.
\end{proof}

\begin{lemma}
\label{lem 3.2}Let $f_{1},f_{2},f_{3}:\mathbb{Z}_{15}^{\ast}\rightarrow
\lbrack0,1]$ be arbitrary functions satisfying
\[
F_{1}F_{2}+F_{2}F_{3}+F_{3}F_{1}>5(F_{1}+F_{2}+F_{3}),
\]
where $F_{i}=%
{\displaystyle\sum\limits_{x\in\mathbb{Z}_{15}^{\ast}}}
f_{i}(x)$. Then for every $x\in\mathbb{Z}_{15}$, there exist $a,b,c\in
\mathbb{Z}_{15}^{\ast}$ with $x=a+b+c$, such that%
\[
f_{1}(a)f_{2}(b)f_{3}(c)>0,\text{ }f_{1}(a)+f_{2}(b)+f_{3}(c)>\frac{3}{2}.
\]

\end{lemma}

\begin{proof}
See \cite[Proposition 3.2]{5 X. Shao}.
\end{proof}

\textbf{\ }Now we deduce Theorem \ref{thm 1.3}. First note that if the result
is true for $m$, then it holds for any $m^{^{\prime}}$ dividing $m$. So we
suppose $15|m$. Write $m=15m^{^{\prime}}$. Note that $(m^{^{\prime}},30)=1$.
Since $\mathbb{Z}_{m}\cong\mathbb{Z}_{m^{^{\prime}}}\times\mathbb{Z}_{15}$, we
can write $(u,v)$ $(u\in\mathbb{Z}_{m^{^{\prime}}},v\in\mathbb{Z}_{15})$ as
the arbitrary term in $\mathbb{Z}_{m}$. Define $f_{i}^{^{\prime}}%
:\mathbb{Z}_{m^{^{\prime}}}\rightarrow\lbrack0,1]$ $(i=1,2,3)$ by%
\[
f_{i}^{^{\prime}}(x)=\frac{1}{\phi(15)}%
{\displaystyle\sum\limits_{y\in\mathbb{Z}_{15}^{\ast}}}
f_{i}(x,y).
\]
Note that $f_{i}^{^{\prime}}(x)$ $(i=1,2,3)$ satisfy the condition of Lemma
\ref{lem 3.1}, and we can conclude that for every $u\in\mathbb{Z}%
_{m^{^{\prime}}}$, there exist $a_{1},a_{2},a_{3}\in\mathbb{Z}_{m^{^{\prime}}%
}^{\ast}$ with $u=a_{1}+a_{2}+a_{3}$, such that%
\begin{equation}
f_{1}^{^{\prime}}(a_{1})f_{2}^{^{\prime}}(a_{2})+f_{2}^{^{\prime}}(a_{2}%
)f_{3}^{^{\prime}}(a_{3})+f_{3}^{^{\prime}}(a_{3})f_{1}^{^{\prime}}%
(a_{1})>\frac{5}{8}(f_{1}^{^{\prime}}(a_{1})+f_{2}^{^{\prime}}(a_{2}%
)+f_{3}^{^{\prime}}(a_{3})). \label{equ:3.3}%
\end{equation}
Now define $f_{i}^{\#}:\mathbb{Z}_{15}^{\ast}\rightarrow\lbrack0,1]$ by%
\[
f_{i}^{\#}(y)=f_{i}(a_{i},y).
\]
With (\ref{equ:3.3}), we note that $f_{i}^{\#}(y)$ satisfy the condition of
Lemma \ref{lem 3.2}. Thus, for every $v\in\mathbb{Z}_{15}$, there exist
$b_{1},b_{2},b_{3}\in\mathbb{Z}_{15}^{\ast}$ with $v=b_{1}+b_{2}+b_{3}$, such
that
\[
f_{i}^{\#}(b_{i})>0\text{, }f_{1}^{\#}(b_{1})+f_{2}^{\#}(b_{2})+f_{3}%
^{\#}(b_{3})>\frac{3}{2}.
\]
Note that $(u,v)=(a_{1},b_{1})+(a_{2},b_{2})+(a_{3},b_{3})$. It follows that
\[
f_{i}(a_{i},b_{i})>0\text{ }(i=1,2,3)\text{, }f_{1}(a_{1},b_{1})+f_{2}%
(a_{2},b_{2})+f_{3}(a_{3},b_{3})>\frac{3}{2}.
\]
This completes the proof.

\section{Sketch of the proof of Theorem 1.1}

The proof is almost same as in \cite{5 X. Shao}. Therefore, we omit the
details. Theorem \ref{thm 1.1} can be deduced from the following transference
principle Proposition \ref{prop 3.3}.

For $f:\mathbb{Z}_{N}\rightarrow\mathbb{C}$, we define the Fourier transform
of $f$ by%
\[
f(r)=%
{\displaystyle\sum\limits_{x\in\mathbb{Z}_{N}}}
f(x)e_{N}(rx),\text{ }r\in\mathbb{Z}_{N},
\]
where $e_{N}(y)=\exp(2\pi iy/N)$.

\begin{proposition}
\label{prop 3.3}Let $N$ be a sufficiently large prime. Suppose that $\mu
_{i}:\mathbb{Z}_{N}\rightarrow\mathbb{R}^{+}$ and $a_{i}:\mathbb{Z}%
_{N}\rightarrow\mathbb{R}^{+}$ $(i=1,2,3)$ are functions satisfying the
majorization condition
\[
0\leq a_{i}(n)\leq\mu_{i}(n),
\]
and the mean condition
\[
\min(\delta_{1},\delta_{2},\delta_{3},\delta_{1}+\delta_{2}+\delta_{3}%
-1)\geq\delta
\]
for some $\delta>0$, where $\delta_{i}=%
{\displaystyle\sum\nolimits_{x\in\mathbb{Z}_{N}}}
a_{i}(x)$ $(i=1,2,3)$. Suppose that $\mu_{i}$ and $a_{i}$ also satisfy the
pseudorandomness conditions$\ $%
\[
|\hat{\mu}_{i}(r)-\delta_{r,0}|\leq\eta,\text{ }r\in\mathbb{Z}_{N},
\]
where $\delta_{r,0}$ is the Kronecker delta, and%
\[
||\hat{a}_{i}||_{q}=\left(
{\displaystyle\sum\limits_{r\in\mathbb{Z}_{N}}}
|\hat{a}_{i}(r)|^{q}\right)  ^{1/q}\leq M
\]
for some $2<q<3$ and $\eta,M>0$. Then for any $x\in\mathbb{Z}_{N}$, we have
\[%
{\displaystyle\sum\limits_{y,z\in\mathbb{Z}_{N}}}
a_{1}(y)a_{2}(z)a_{3}(x-y-z)\geq\frac{c(\delta)}{N}%
\]
for some constant $c(\delta)>0$ depending only on $\delta$, provided that
$\eta\leq\eta(\delta,M,q)$ is small enough.
\end{proposition}

\begin{proof}
See \cite[Proposition 4.1]{5 X. Shao}.
\end{proof}

Let $n$ be a very large positive odd integer. The aim is to show there exist
$p_{1}\in P_{1}$, $p_{2}\in P_{2}$, and $p_{3}\in P_{3}$ such that
$n=p_{1}+p_{2}+p_{3}$. In the case of Theorem \ref{thm 1.1}, we note that
there exist $0<\delta<5/12$ and $0<\eta<\delta/50$ such that%
\begin{align}
\left\vert P_{1}\cap\lbrack1,N]\right\vert  &  >(\frac{5}{8}+\delta)\frac
{N}{\log N},\nonumber\\
\left\vert P_{i}\cap\lbrack1,N]\right\vert  &  >(\frac{5}{8}-\eta)\frac
{N}{\log N}\text{ }(i=2,3), \label{equ:4.1}%
\end{align}
for any sufficiently large integer $N>0$. Define $f_{i}:\mathbb{Z}_{W}^{\ast
}\rightarrow\lbrack0,1]$ $(i=1,2,3)$ by%
\[
f_{i}(b)=\max\left(  \frac{3\phi(W)}{2n}%
{\displaystyle\sum\limits_{x\in P_{i}\cap(W\mathbb{Z}+b),x<\frac{2n}{3}}}
\log x-\frac{\delta}{8},0\right)  .
\]
Here $W=%
{\displaystyle\prod\nolimits_{p\text{ prime, }p<z}}
p$, where $z=z(\delta)$ is a large parameter. It follows from (\ref{equ:4.1}) that%

\begin{align*}%
{\displaystyle\sum\limits_{b\in\mathbb{Z}_{W}^{\ast}}}
f_{1}(b)  &  >(\frac{5}{8}+\frac{3\delta}{8})\phi(W),\\
\text{ }%
{\displaystyle\sum\limits_{b\in\mathbb{Z}_{W}^{\ast}}}
f_{i}(b)  &  >(\frac{5}{8}-(\frac{5\eta}{4}+\frac{\delta}{8}))\phi(W)\text{
}(i=2,3).
\end{align*}
Note that $\frac{5\eta}{4}+\frac{\delta}{8}<\frac{2}{5}\times\frac{3\delta}%
{8}$ by $0<\eta<\delta/50$. We can deduce from Theorem \ref{thm 1.3} that
there exist $b_{1},b_{2},b_{3}\in\mathbb{Z}_{W}^{\ast}$ with $b_{1}%
+b_{2}+b_{3}\equiv n$ $(\operatorname{mod}W)$ such that
\begin{equation}
f_{1}(b_{1})f_{2}(b_{2})f_{3}(b_{3})>0,f_{1}(b_{1})+f_{2}(b_{2})+f_{3}%
(b_{3})>\frac{3}{2}. \label{equ:4.3}%
\end{equation}
The rest part of the proof is just like the proof in \cite{5 X. Shao}.
Applying (\ref{equ:4.3}), one can confirm the mean condition in Proposition
\ref{prop 3.3}. The pseudorandomness conditions hold by Lemma 6.2 and Lemma
6.6 in \cite{2 Green}. The majorization condition is satisfied immediately
from the definitions of $a_{i}$ and $\mu_{i}$. Then the transference principle
is applied, leading to Theorem \ref{thm 1.1}. Here we want to refer readers to
section 4 of \cite{5 X. Shao} for further details.

\textbf{Acknowledgements. }The author would like to thank his advisor
Professor Yonghui Wang specially for his constant guidance, and Wenying Chen
for helpful discussions in seminar.

\textbf{Contact information:}\newline Quanli Shen\newline Department of Math,
Capital Normal University\newline Xi San Huan Beilu 105, Beijing 100048, P.R.
China,\newline Email: qlshen@outlook.com.

\end{document}